\documentclass[12pt,a4paper]{article}
\usepackage[latin 2]{inputenc}
\usepackage{amssymb}
\usepackage{latexsym}
\usepackage{amsmath}
\usepackage{amsthm}
\newtheorem{thm}{Theorem}

\newtheorem{lemma}[thm]{Lemma}
\newtheorem{cor}[thm]{Corollary}
\theoremstyle{definition}
\usepackage{xpatch}
\xpatchcmd{\proof}{\itshape}{\normalfont\proofnameformat}{}{}
\newcommand{\proofnameformat}{}
\begin{document}

\renewcommand{\proofnameformat}{\bfseries}

\begin{center}
{\large\textbf{Bounded error uniformity of the linear flow on the torus}}

\vspace{5mm}

\textbf{Bence Borda}

{\footnotesize Alfr\'ed R\'enyi Institute of Mathematics, Hungarian Academy of Sciences

1053 Budapest, Re\'altanoda u.\ 13--15, Hungary

Email: \texttt{bordabence85@gmail.com}}

\vspace{5mm}


\end{center}

\vspace{5mm}

\begin{abstract}
A linear flow on the torus $\mathbb{R}^d / \mathbb{Z}^d$ is uniformly distributed in the Weyl sense if the direction of the flow has linearly independent coordinates over $\mathbb{Q}$. In this paper we combine Fourier analysis and the subspace theorem of Schmidt to prove bounded error uniformity of linear flows with respect to certain polytopes if, in addition, the coordinates of the direction are all algebraic. In particular, we show that there is no van Aardenne--Ehrenfest type theorem for the mod $1$ discrepancy of continuous curves in any dimension, demonstrating a fundamental difference between continuous and discrete uniform distribution theory.
\end{abstract}

\section{Introduction}
\label{Introduction}

Arguably the simplest continuous time dynamical system on the $d$-dimensional torus $\mathbb{R}^d / \mathbb{Z}^d$ is the \textit{linear flow}: given $\alpha \in \mathbb{R}^d$, a point $s \in \mathbb{R}^d / \mathbb{Z}^d$ is mapped to $s+t\alpha \pmod{\mathbb{Z}^d}$ at time $t \in \mathbb{R}$. We call $\alpha$ the direction of the linear flow (although we do not assume $\alpha$ to have unit norm). The classical theorem of Kronecker on simultaneous Diophantine approximation shows that the linear flow with direction $\alpha = (\alpha_1, \dots, \alpha_d) \in \mathbb{R}^d$ is minimal, that is, every orbit is dense in $\mathbb{R}^d / \mathbb{Z}^d$ if and only if the coordinates $\alpha_1, \dots, \alpha_d$ are linearly independent over $\mathbb{Q}$. A stronger result was later obtained by Weyl \cite{Weyl}. As an application of the famous Weyl's criterion he proved that the linear flow with direction $\alpha$ is \textit{uniformly distributed} if and only if the same linear independence condition holds.

To define what we mean by uniform distribution, let us work in the fundamental domain $[0,1]^d$ (where the opposite facets are identified). Fixing a starting point $s \in [0,1]^d$, the flow is thus given by the parametrized curve $(\{s_1+t \alpha_1 \}, \dots, \{ s_d +t \alpha_d \})$, $t \in \mathbb{R}$, where $\{ \cdot \}$ denotes fractional part. For a function $f: [0,1]^d \to \mathbb{R}$ let
\begin{equation}\label{DeltaT}
\Delta_T (s, \alpha, f)=\int_0^T f(\{ s_1+t\alpha_1 \}, \dots, \{ s_d+t\alpha_d \}) \, \mathrm{d}t - T \int_{[0,1]^d} f(x) \, \mathrm{d}x  \qquad (T>0).
\end{equation}
In the terminology of dynamical systems $\Delta_T(s, \alpha, f)/T$ is the difference of the ``time average'' and the ``space average'' of $f$ along the orbit of $s$. For a set $A \subseteq [0,1]^d$ let $\chi_A$ denote its characteristic function, and put $\Delta_T (s, \alpha, A)=\Delta_T (s, \alpha, \chi_A)$. We say that the linear flow with direction $\alpha$ is uniformly distributed if for any starting point $s \in [0,1]^d$ and any axis parallel box $R=\prod_{k=1}^d[a_k, b_k] \subseteq [0,1]^d$ we have $\Delta_T(s, \alpha, R)=o(T)$, i.e.\ $\lim_{T \to \infty} \Delta_T(s, \alpha, R)/T=0$. Note that we would have an equivalent definition by using polytopes, or even arbitrary convex sets instead of axis parallel boxes. Alternatively, we could define uniform distribution by stipulating that $\Delta_T (s, \alpha, f)=o(T)$ for any starting point $s \in [0,1]^d$ and any continuous function $f:[0,1]^d \to \mathbb{R}$. For the theory of uniform distribution of continuous curves we refer the reader to \cite[Chapter 2.3]{Drmotabook}.

It is also well-known that the linear flow with direction $\alpha=(\alpha_1, \dots, \alpha_d) \in \mathbb{R}^d$ is ergodic with respect to the Haar measure on $\mathbb{R}^d / \mathbb{Z}^d$ (which coincides with the Lebesgue measure on the fundamental domain $[0,1]^d$) if and only if the coordinates $\alpha_1, \dots, \alpha_d$ are linearly independent over $\mathbb{Q}$ \cite[Chapter 3.1]{Sinai}. The ergodicity allows us to study $\Delta_T(s, \alpha, f)$ for more general test functions $f$. Most importantly, by Birkhoff's pointwise ergodic theorem \cite[Chapter 1.2]{Sinai}, for any Lebesgue integrable function $f \in L^1([0,1]^d)$ we have $\Delta_T (s, \alpha, f)=o(T)$ for almost every $s \in [0,1]^d$. In particular, for any Lebesgue measurable set $A \subseteq [0,1]^d$ we have $\Delta_T(s, \alpha, A)=o(T)$ for almost every $s \in [0,1]^d$.

The minimality, the uniform distribution and the ergodicity of a linear flow on $\mathbb{R}^d /\mathbb{Z}^d$ are thus all equivalent. This remarkable fact can actually be generalized to flows generated by a continuous one-parameter subgroup of an arbitrary compact Abelian group \cite[Chapter 4.1]{Sinai}. Moreover, the linear independence condition also has an analogue in terms of the characters of the group.

A common aspect of Weyl's criterion and Birkhoff's pointwise ergodic theorem is that for certain classes of test functions $f$ they only yield $\Delta_T(s, \alpha, f)=o(T)$ without an estimate on the rate of convergence. A quantitative form of ergodicity was obtained by Beck \cite{Beck}: given a function $f \in L^2([0,1]^d)$, for almost every unit vector $\alpha \in \mathbb{R}^d$, $|\alpha|=1$ (in the sense of the $(d-1)$-dimensional Hausdorff measure on the unit sphere in $\mathbb{R}^d$) we have $\Delta_T(0,\alpha,f)=o(T^{1/2-1/(2d-2)} \log^{3+\varepsilon} T)$ for any $\varepsilon>0$. Moreover, the estimate is almost tight in the sense that the result does not hold with $o(T^{1/2-1/(2d-2)})$. Note that the starting point is the origin. In particular, the result applies to $f=\chi_A$ with an arbitrary Lebesgue measurable set $A \subseteq [0,1]^d$. It is interesting to note that in dimension $d=2$ the estimate is simply $O(\log^{3+\varepsilon} T)$. To describe this phenomenon, i.e.\ uniformity with polylogarithmic error, Beck introduced the term \textit{superuniformity}. The main message is thus that for the family of all Lebesgue measurable test sets we have superuniformity in dimension $d=2$ but not in dimensions $d \ge 3$.

For a more narrow class of test sets, however, we can improve superuniformity to \textit{bounded error uniformity}. Such results have only been proved in dimension $d=2$ so far. Let $\left\| \cdot \right\|$ denote the distance from the nearest integer function. For the sake of simplicity, let us only consider directions of the form $\alpha=(\alpha_1, 1)$. Drmota \cite{Drmotapaper} showed that if there exists a constant $\eta<2$ such that the inequality $\left\| n \alpha_1 \right\|<|n|^{-\eta}$ has finitely many integer solutions $n \in \mathbb{Z}$, then for any axis parallel box $R \subseteq [0,1]^2$ we have $\Delta_T (0, \alpha, R)=O(1)$. In fact, the implied constant depends only on $\alpha$, which means that by letting $\mathcal{R}$ denote the family of axis parallel boxes in $[0,1]^2$, the \textit{discrepancy} $\sup_{R \in \mathcal{R}} |\Delta_T (0, \alpha, R)|$ is also $O(1)$. Grepstad and Larcher \cite{Grepstad} considered convex polygons $P \subseteq [0,1]^2$ with no side parallel to the direction $\alpha=(\alpha_1, 1)$ as test sets. If the continued fraction representation $\alpha_1=[a_0;a_1,a_2, \dots]$ satisfies $\sum_{\ell=0}^{\infty} a_{\ell+1}/q_{\ell}^{1/2} \sum_{k=1}^{\ell+1}a_k < \infty$, where $p_{\ell}/q_{\ell}=[a_0;a_1, \dots, a_{\ell}]$ denotes the convergents to $\alpha_1$, then for any starting point $s \in [0,1]^2$ we have $\Delta_T(s, \alpha, P)=O(1)$. To make the two results easier to compare let us mention that the condition on the continued fraction holds if there exists a constant $\eta<5/4$ such that the inequality $\left\| n \alpha_1 \right\|<|n|^{-\eta}$ has finitely many integer solutions $n \in \mathbb{Z}$. Both results are tight: the estimate $O(1)$ clearly cannot be replaced by $o(1)$ in either theorem. See, however, Theorem \ref{theorem4} below for an explicit bound.

It is natural to ask what the widest class of test sets is for which we have bounded error uniformity. Well, for the family of all convex test sets in $[0,1]^2$ we have superuniformity, but not bounded error uniformity. More precisely, Beck \cite{Beck2} proved for the direction $\alpha=(\alpha_1,1)$ that if the continued fraction representation $\alpha_1=[a_0;a_1,a_2, \dots]$ satisfies $a_{\ell}=O(1)$ (i.e.\ $\alpha_1$ is badly approximable), then for any convex set $C \subseteq [0,1]^2$ we have $\Delta_T (0, \alpha, C)=O(\log T)$. In fact, the implied constant depends only on $\alpha$, thus the \textit{isotropic discrepancy} $\sup_C |\Delta_T (0, \alpha, C)|$, where the supremum is taken over all convex sets $C \subseteq [0,1]^2$ is also $O(\log T)$. Moreover, the estimate is tight. In light of Grepstad and Larcher's theorem it is not surprising that the convex set showing that $O(\log T)$ cannot be replaced by $o(\log T)$ is a parallelogram with two sides parallel to $\alpha$.

To summarize, for arbitrary Lebesgue measurable test sets we only have metric results, that is, the estimates only hold for almost every direction $\alpha$ (but the starting point can be specified). On the other hand, for simple test sets, like boxes, polygons or convex sets in dimension $d=2$, we have quantitative uniformity results for explicit directions $\alpha$ and starting points $s$. Indeed, Beck's result on the isotropic discrepancy holds in particular for directions $\alpha=(\alpha_1,1)$ with quadratic irrational $\alpha_1$, say $\alpha_1=\sqrt{2}$. The theorems of Drmota, and Grepstad and Larcher hold for even more general directions, e.g.\ for algebraic irrational $\alpha_1$: recall that the classical theorem of Roth \cite{Roth} states that if $\alpha_1$ is an algebraic irrational, then for any $\varepsilon>0$ the inequality $\left\| n \alpha_1 \right\| < |n|^{-1-\varepsilon}$ has finitely many integer solutions $n \in \mathbb{Z}$. Thus we have a wide class of explicit directions for which the estimates are valid.

The main purpose of this paper is to prove bounded error uniformity results in arbitrary dimensions $d \ge 2$. Our test sets will be polytopes, i.e.\ convex hulls of finitely many points. The $(d-1)$-dimensional faces of a polytope will be called facets; by a normal vector of a facet we mean a nonzero vector, not necessarily of unit norm, which is orthogonal to the facet. Let $|x|$ denote the Euclidean norm, and $\langle x,y \rangle=\sum_{k=1}^d x_k y_k$ the scalar product of $x,y \in \mathbb{R}^d$, and let $\lambda$ be the Lebesgue measure. The notation $f(T)=O(g(T))$ means that there exists an (implied) constant $K>0$ such that $|f(T)| \le K g(T)$ for every $T>0$. We say that $f(T)=\Omega (g(T))$ if $\limsup_{T \to \infty} |f(T)|/g(T)>0$. Similar notations are used for sequences. The following bounded error uniformity result holds for explicit directions and starting points in arbitrary dimension.
\begin{thm}\label{theorem1} Let $d \ge 2$, and suppose that the coordinates of $\alpha=(\alpha_1, \dots, \alpha_d) \in \mathbb{R}^d$ are algebraic and linearly independent over $\mathbb{Q}$. Let $P \subseteq [0,1]^d$ be a polytope with a nonempty interior, and suppose that every facet of $P$ has a normal vector $\nu$ with algebraic coordinates and $\langle \nu, \alpha \rangle \neq 0$. For any starting point $s \in [0,1]^d$
\[ \Delta_T (s, \alpha, P)=O(1) \]
with an implied constant depending only on $\alpha$ and the normal vectors of the facets of $P$.
\end{thm}
Clearly, for any $\alpha \in \mathbb{R}^d$, any $s \in [0,1]^d$ and any polytope $P \subseteq [0,1]^d$ with $0<\lambda (P)<1$ we have $\Delta_T(s, \alpha, P)=\Omega(1)$, therefore the estimate in Theorem \ref{theorem1} is best possible. It is interesting to note that the implied constant does not depend on $P$ itself, only on the normal vectors of its facets. This means that if $P$ is a polytope satisfying the conditions of Theorem \ref{theorem1}, then we actually have a uniform estimate for all test sets of the form $aP+b \subseteq [0,1]^d$, where $a>0$ and $b \in \mathbb{R}^d$. Furthermore, note that for axis parallel boxes the normal vectors of the facets are all $\pm 1$ times a standard basis vector of $\mathbb{R}^d$, thus we immediately obtain a corollary on the discrepancy.
\begin{cor}\label{corollary2} Let $d \ge 2$, and suppose that the coordinates of $\alpha=(\alpha_1, \dots, \alpha_d) \in \mathbb{R}^d$ are algebraic and linearly independent over $\mathbb{Q}$. For any starting point $s \in [0,1]^d$
\[ \sup_{R \in \mathcal{R}} |\Delta_T (s, \alpha, R)| =O(1) \]
with an implied constant depending only on $\alpha$, where $\mathcal{R}$ denotes the family of axis parallel boxes in $[0,1]^d$.
\end{cor}

A comparison with the corresponding discrete problem is in order. Given $\alpha \in \mathbb{R}^d$, the discrete analogue of the linear flow with direction $\alpha$ is the \textit{translation} with direction $\alpha$, that is, the discrete time dynamical system on $\mathbb{R}^d / \mathbb{Z}^d$ in which a point $s \in \mathbb{R}^d/\mathbb{Z}^d$ is mapped to $s+k\alpha \pmod{\mathbb{Z}^d}$ at time $k \in \mathbb{Z}$. The analogue of \eqref{DeltaT} is of course
\begin{equation}\label{DN}
D_N(s, \alpha, f)=\sum_{k=0}^{N-1} f(\{s_1 + k\alpha_1 \}, \dots, \{ s_d + k\alpha_d \} ) -N \int_{[0,1]^d} f(x) \, \textrm{d}x \qquad (N \in \mathbb{N}),
\end{equation}
and similarly let $D_N (s, \alpha, A)=D_N(s, \alpha, \chi_A)$. We say that the translation with direction $\alpha$ is uniformly distributed if for any starting point $s \in [0,1]^d$ and any axis parallel box $R \subseteq [0,1]^d$ we have $D_N (s, \alpha, R)=o(N)$. Again, we would get an equivalent definition by using polytopes or arbitrary convex sets instead of axis parallel boxes, or by stipulating that $D_N(s, \alpha, f)=o(N)$ for any $s\in [0,1]^d$ and any continuous function $f:[0,1]^d \to \mathbb{R}$.

Similarly to the continuous time case, the minimality, the uniform distribution and the ergodicity of a translation with direction $\alpha=(\alpha_1, \dots, \alpha_d)\in \mathbb{R}^d$ are all equivalent. The only difference is that in the discrete time case these properties hold if and only if $\alpha_1, \dots, \alpha_d, 1$ are linearly independent over $\mathbb{Q}$ \cite[Chapter 3.1]{Sinai}. Again, this fact can actually be generalized to translations on an arbitrary compact Abelian group, with the linear independence condition replaced by a condition in terms of the characters of the group \cite[Chapter 4.1]{Sinai}.

The quantitative results are, however, very different from the continuous time case. Based on the analogy with the linear flow, one could think that given an arbitrary Lebesgue measurable set $A \subseteq [0,1]^d$, for almost every $\alpha \in \mathbb{R}^d$ we have $D_N(0,\alpha,A)=o(N)$. In fact, in dimension $d=1$ this was a famous, long-standing conjecture of Khinchin. Khinchin's conjecture, however, was disproved by Marstrand \cite{Marstrand}, who showed the existence of an open set $A \subseteq [0,1]$ for which $D_N (0,\alpha,A)=\Omega (N)$ for all $\alpha \in \mathbb{R}$. The discrete analogue of Corollary \ref{corollary2} is due to Niederreiter \cite{Niederreiter}: if $\alpha_1, \dots, \alpha_d,1$ are algebraic and linearly independent over $\mathbb{Q}$, then $\sup_{R \in \mathcal{R}} |D_N(0,\alpha,R)|=O(N^{\varepsilon})$ for any $\varepsilon>0$.

Finally, let us mention another, arguably the most important difference between the continuous and the discrete time case. Let us generalize \eqref{DeltaT} and \eqref{DN} as follows: for a continuous curve $g=(g_1, \dots, g_d): [0,\infty) \to \mathbb{R}^d$ let
\[ \Delta_T (g,f)=\int_0^T f(\{ g_1 (t)\}, \dots, \{ g_d (t) \}) \, \textrm{d}t - T\int_{[0,1]^d} f(x) \, \textrm{d}x \qquad (T>0), \]
and similarly, for a sequence $x_k=(x_{k,1}, \dots, x_{k,d}) \in \mathbb{R}^d$ let
\[ D_N (x_k,f)=\sum_{k=0}^{N-1} f(\{ x_{k,1} \}, \dots, \{ x_{k,d} \} )-N \int_{[0,1]^d} f(x) \, \textrm{d}x \qquad (N \in \mathbb{N}). \]
As before, for a set $A \subseteq [0,1]^d$ let $\Delta_T (g,A)= \Delta_T (g, \chi_A)$ and $D_N (x_k,A)=D_N (x_k, \chi_A)$. Note that $g$ and $x_k$ do not necessarily come from dynamical systems. The main difference between continuous and discrete uniform distribution is that bounded error uniformity is impossible in the discrete case, even for the family of axis parallel boxes as test sets. Indeed, answering a question of van der Corput, it was van Aardenne--Ehrenfest \cite{Aardenne} who first proved that in dimension $d=1$, for any sequence $x_k \in \mathbb{R}$ the discrepancy $\sup_{R \in \mathcal{R}} |D_N (x_k,R)|$ cannot be $O(1)$. This was later improved by Schmidt and Roth, who showed that for an arbitrary sequence $x_k \in \mathbb{R}^d$ we have $\sup_{R \in \mathcal{R}}|D_N(x_k,R)| =\Omega (\log N)$ if $d=1$, and $\sup_{R \in \mathcal{R}}|D_N(x_k,R)| =\Omega (\log^{d/2} N)$ if $d \ge 2$, with implied constants depending only on $d$ (see e.g.\ \cite[Chapter 1.3]{Drmotabook}). Similar lower estimates for continuous curves were considered plausible. In particular, Drmota conjectured \cite[eq.\ (121)]{Drmotapaper} that for any continuous curve $g: [0,\infty) \to \mathbb{R}^d$ such that the arc length $\ell_T$ of $g$ on $[0,T]$ is finite for every $T>0$ we have $\sup_{R \in \mathcal{R}} |\Delta_T(g,R)/T|=\Omega ((\log \ell_T )^{d-2-\varepsilon} / \ell_T)$ for any $\varepsilon >0$. The main message of Corollary \ref{corollary2} is thus that there is no van Aardenne--Ehrenfest type theorem for continuous curves in any dimension. In particular, the conjecture of Drmota is false.

\section{The main result}

For the sake of simplicity, let us consider directions $\alpha=(\alpha_1, \dots, \alpha_d) \in \mathbb{R}^d$ such that $\alpha_d=1$. The coordinates $\alpha_1, \dots, \alpha_{d-1},1$ are linearly independent over $\mathbb{Q}$ if and only if $\left\| n_1 \alpha_1 + \cdots + n_{d-1} \alpha_{d-1} \right\|>0$ for every $n \in \mathbb{Z}^{d-1}$, $n \neq 0$. Our most general result is based on the idea that by assuming a stronger, quantitative form of linear independence we can obtain a stronger, quantitative form of uniform distribution.
\begin{thm}\label{maintheorem} Let $d \ge 2$, let $K$ be a subfield of $\mathbb{R}$, and let $\alpha \in K^d$ with $\alpha_d=1$. Suppose that for any linearly independent linear forms $L_1, \dots, L_{d-1}$ of $d-1$ variables with coefficients in $K$ there exists a constant $\gamma <1$ such that the inequality
\[ \left\| \alpha_1 n_1 + \cdots + \alpha_{d-1} n_{d-1} \right\| \cdot \prod_{k=1}^{d-1} \left( |L_k (n)| +1 \right) < |n|^{-\gamma} \]
has finitely many integral solutions $n \in \mathbb{Z}^{d-1}$. Let $P \subseteq [0,1]^d$ be a polytope with a nonempty interior, and suppose that every facet of $P$ has a normal vector $\nu$ with coordinates in $K$ and $\langle \nu, \alpha \rangle \neq 0$. For any starting point $s \in [0,1]^d$
\[ \Delta_T (s, \alpha, P) = O(1) \]
with an implied constant depending only on $\alpha$ and the normal vectors of the facets of $P$.
\end{thm}
In dimension $d=2$ there is only one linear form of $d-1=1$ variable up to a constant factor, while in higher dimensions there are infinitely many. This fact makes it easier to obtain an explicit bound in the case $d=2$ as follows.
\begin{thm}\label{theorem4} Let $\alpha=(\alpha_1,1) \in \mathbb{R}^2$ be such that $0<\alpha_1<1$ is irrational, and let $P \subseteq [0,1]^2$ be a convex polygon with edges $e_1, e_2, \ldots, e_N$. Suppose that none of the edges of $P$ are parallel to $\alpha$, and for every $1 \le k \le N$ let $\phi_k$ denote the angle such that $\alpha$ rotated by $\phi_k$ in the positive direction is parallel to $e_k$. For any starting point $s \in [0,1]^2$ and any $T>0$ we have
\[ |\Delta_T (s, \alpha, P)| \le 2 + \frac{N+1}{\pi^2 |\alpha |} \max_{1 \le k<\ell \le N} \left| \cot \phi_k - \cot \phi_{\ell} \right| \sum_{n=1}^{\infty} \frac{1}{n^2 \| n \alpha_1 \|}. \]
\end{thm}
By switching the coordinates if necessary, we may assume that the slope of the orbits is greater than 1, therefore the assumption $0<\alpha_1<1$ is not restrictive. The proof will clearly show that if the second coordinate of $s$ is $0$ and $T \in \mathbb{N}$, then the estimate in Theorem \ref{theorem4} holds even without the first term $2$. Note that if there exists a constant $\eta<2$ such that the inequality $\| n \alpha_1 \|<|n|^{-\eta}$ has finitely many integer solutions $n \in \mathbb{Z}$, then $\sum_{n=1}^{\infty} 1/(n^2 \| n \alpha_1 \|)<\infty$.

The rest of this Section is devoted to the proofs of Theorems \ref{maintheorem} and \ref{theorem4}, both of which are based on Fourier analysis. We deduce Theorem \ref{theorem1} from Theorem \ref{maintheorem} and the subspace theorem of Schmidt in Section \ref{section3}.

\begin{proof}[Proof of Theorem \ref{maintheorem}] Throughout this proof the implied constants in the $O$-notation will only depend on $\alpha$ and the normal vectors of the facets of $P$. The error of replacing $s$ by $(\{s_1-\alpha_1 s_d\}, \dots , \{s_{d-1}-\alpha_{d-1} s_d\}, 0)$, and $T$ by $\lceil T \rceil$ in $\Delta_T(s, \alpha, P)$ is clearly $O(1)$, therefore we may assume $s_d=0$, and that $T$ is a positive integer. We start by reducing our $d$-dimensional, continuous time dynamical system to a $(d-1)$-dimensional, discrete time one. By breaking up the integral in the definition of $\Delta_T (s, \alpha, P)$ we get
\[ \Delta_T(s, \alpha, P) = \sum_{k=0}^{T-1} \left( \int_k^{k+1} \chi_P (\{ s_1+t \alpha_1 \}, \dots, \{ s_{d-1}+t \alpha_{d-1} \}, \{ t \} ) \, \mathrm{d}t - \lambda (P) \right) . \]
Applying the integral transformation $t \mapsto t+k$ we can write $\Delta_T (s, \alpha, P)$ in the form
\begin{equation}\label{discretedynsyst}
\Delta_T(s, \alpha, P) = \sum_{k=0}^{T-1} \left( f(s_1+ k \alpha_1, \dots, s_{d-1}+ k \alpha_{d-1}) - \lambda (P) \right) ,
\end{equation}
where $f: \mathbb{R}^{d-1} \to \mathbb{R}$ is defined as
\begin{equation}\label{fdefinition}
f(x_1, \dots, x_{d-1}) = \int_0^1 \chi_P (\{ x_1 + t \alpha_1 \} , \dots , \{ x_{d-1} + t \alpha_{d-1} \} , t) \, \mathrm{d}t .
\end{equation}
In the terminology of dynamical systems the facet $x_d=0$ of $[0,1]^d$ (which corresponds to a $(d-1)$-dimensional torus in $\mathbb{R}^d / \mathbb{Z}^d$) is a transversal, and the underlying discrete time dynamical system, the translation on $\mathbb{R}^{d-1} / \mathbb{Z}^{d-1}$ with direction $(\alpha_1, \dots, \alpha_{d-1})$ is a Poincar\'e map.

The geometric meaning of $f$ is the following. Consider the line segment starting at the point $(x_1, \dots, x_{d-1},0)$ parallel to $\alpha$, joining the facets $x_d=0$ and $x_d=1$ of $[0,1]^d$ (of course everything is taken modulo $\mathbb{Z}^d$, i.e.\ it is in fact a line segment on the torus). Then $f(x_1, \dots, x_{d-1})$ is the length of the intersection of this line segment with $P$. The crucial observation is that since $\alpha$ is not parallel to any facet of $P$, the function $f$ is continuous. This allows us to prove a nontrivial estimate for the Fourier coefficients of $f$ as follows.

\begin{lemma}\label{fourierlemma} There exists a set $\mathcal{L}$ of linearly independent linear forms $(L_1, \dots, L_{d-1})$ of $d-1$ variables with coefficients in $K$, depending only on $\alpha$ and the normal vectors of the facets of $P$, such that $|\mathcal{L}|=O(1)$ and for any $n \in \mathbb{Z}^{d-1}$, $n \neq 0$ we have
\[ \int_{[0,1]^{d-1}} f(x) e^{- 2 \pi i \langle n,x \rangle} \, \mathrm{d}x = O \left( \sum_{(L_1, \dots, L_{d-1}) \in \mathcal{L}} \frac{1}{|n| \prod_{k=1}^{d-1} \left( |L_k (n)|+1 \right)} \right) . \]
\end{lemma}

\begin{proof} We start by ``lifting'' the line segment in the definition of $f$ from $\mathbb{R}^d / \mathbb{Z}^d$ to $\mathbb{R}^d$. For a given $x \in [0,1]^{d-1}$ let $g_x(t)=(x_1+t \alpha_1, \dots , x_{d-1}+t\alpha_{d-1}, t)$, $t \in \mathbb{R}$ denote a parametrized line. Let $M$ be a positive integer such that $|\alpha_k| \le M$ for all $1 \le k \le d$. For any $x \in [0,1]^{d-1}$ the line segment $g_x(t)$, $t \in [0,1]$ stays in $[-M,M+1]^d$. Thus it is enough to consider the translations of $P$ by the integral vectors $\varepsilon$ in the set $E=[-M, M]^d \cap \mathbb{Z}^d$. Formally, for any $x \in [0,1]^{d-1}$ we have
\begin{equation}\label{P+epsilon}
f(x)= \sum_{\varepsilon \in E} \int_0^1 \chi_{P + \varepsilon} (g_x(t)) \, \mathrm{d}t .
\end{equation}

Note that $|E|=O(1)$. We claim that $f$ is a ``piecewise linear'' function. That is, there exists a decomposition of $[0,1]^{d-1}$ into polytopes $A_1, A_2, \dots, A_m$ such that $f$ is of the form $f(x)=\langle a_j, x \rangle + b_j$ on $A_j$ with some $a_j \in \mathbb{R}^{d-1}, b_j \in \mathbb{R}$.

Indeed, let $\pi : \mathbb{R}^d \to \mathbb{R}^{d-1}$ denote the projection onto the hyperplane $x_d=0$ in the direction $\alpha$, i.e.\ let $\pi (x_1, \dots, x_d)=(x_1-\alpha_1 x_d, x_2-\alpha_2 x_d, \dots, x_{d-1}-\alpha_{d-1} x_d)$. Consider the $(d-2)$-dimensional faces of all translates $P+\varepsilon$, $\varepsilon \in E$. Applying the projection $\pi$ to the affine hulls of these $(d-2)$-dimensional faces, we obtain affine hyperplanes in $\mathbb{R}^{d-1}$. These affine hyperplanes decompose $[0,1]^{d-1}$ into polytopes $A_1, \dots, A_m$. (The affine hyperplanes which do not intersect $[0,1]^{d-1}$ are discarded.) Observe that $m=O(1)$ and that each $A_j$ has $O(1)$ facets. More specifically, consider a $(d-2)$-dimensional face of one of the translates $P+\varepsilon$. The affine hull of this face is the set of solutions of the system $\langle \mu, x \rangle =b$, $\langle \nu , x \rangle =c$ for the normal vectors $\mu, \nu$ of two facets of $P$ and some $b,c \in \mathbb{R}$. The projection $\pi (x)=y$ satisfies
\[ \sum_{k=1}^{d-1} \left( \frac{\mu_k}{\langle \mu, \alpha \rangle} - \frac{\nu_k}{\langle \nu, \alpha \rangle} \right) y_k = \frac{b}{\langle \mu, \alpha \rangle}-\frac{c}{\langle \nu, \alpha \rangle}. \]
Here the coefficients of $y_k$ belong to the field $K$, and it is not difficult to check that they are not all zero. Hence the ($(d-2)$-dimensional) facets of the ($(d-1)$-dimensional) polytopes $A_1, \dots, A_m$ have normal vectors with coefficients in $K$.

For a given $x \in [0,1]^{d-1}$ the intersection of the line segment $g_x(t)$, $t \in [0,1]$ and the polytopes $P+\varepsilon$, $\varepsilon \in E$ is the union of finitely many (possibly zero) line segments with endpoints on the facets of $P+\varepsilon$, $\varepsilon \in E$. Observe that given an $A_j$, the ordered list of facets of $P+\varepsilon$, $\varepsilon \in E$ intersecting $g_x(t)$, $t \in [0,1]$ does not depend on the choice of the point $x \in A_j$.

Fix an $A_j$, and let $x \in A_j$. Consider two facets of $P+\varepsilon$, $\varepsilon \in E$ whose affine hulls have equations $\langle \mu,y \rangle = b$ and $\langle \nu, y \rangle =c$ with normal vectors $\mu, \nu$ and some $b,c \in \mathbb{R}$. The points of the line $g_x(t)$ that lie on these affine hyperplanes satisfy
\begin{equation}\label{ajformula}
t = \frac{b}{\langle \mu, \alpha \rangle} - \sum_{k=1}^{d-1} \frac{\mu_k}{\langle \mu, \alpha \rangle} x_k , \qquad t= \frac{c}{\langle \nu, \alpha \rangle} - \sum_{k=1}^{d-1} \frac{\nu_k}{\langle \nu, \alpha \rangle} x_k,
\end{equation}
respectively. Therefore the length of the line segment on $g_x(t)$ that lies between the two given facets is an inhomogeneous linear function of $x$. Observe also that the coefficients of $x_1, \dots, x_{d-1}$ in this inhomogeneous linear function are $O(1)$. From \eqref{P+epsilon} we thus obtain that $f(x)$ is indeed of the form $f(x)=\langle a_j, x \rangle + b_j$ on $A_j$ with some $a_j \in \mathbb{R}^{d-1}$ and $b_j \in \mathbb{R}$, moreover $|a_j|=O(1)$.

We are interested in the integral of $f(x)e^{-2 \pi i \langle n,x \rangle}$, i.e.\ the product of an inhomogeneous linear, and an exponential function. It is therefore natural to use the divergence theorem, which is basically a multidimensional analogue of integration by parts. The key fact is that the continuity of $f$ (which follows from the assumption that $\alpha$ is not parallel to any facet of $P$) implies that the integrals over the boundaries in the divergence theorem \textit{completely cancel}. The appearance of the extra factor $|n|$ in the denominator in Lemma \ref{fourierlemma}, and hence the boundedness of $\Delta_T(s, \alpha, P)$ is a consequence of this cancellation in the divergence theorem.

From now on let $n \in \mathbb{Z}^{d-1}$, $n \neq 0$ be fixed. For a given $1 \le j \le m$ let us apply the divergence theorem to the function $F: A_j \to \mathbb{R}^{d-1}$,
\[ F(x)=\frac{n}{2 \pi i |n|^2} f(x) e^{-2 \pi i \langle n, x \rangle} = \frac{n}{2 \pi i |n|^2} \left( \langle a_j, x \rangle + b_j \right) e^{-2 \pi i \langle n, x \rangle} \]
to obtain
\begin{equation}\label{divtheorem}
\int_{A_j} \left(\frac{\langle a_j, n \rangle}{2 \pi i |n|^2} e^{-2 \pi i \langle n, x \rangle} -f(x) e^{-2 \pi i \langle n, x \rangle} \right) \, \textrm{d} x = \int_{\partial A_j} \frac{\langle n, \nu (x) \rangle}{2 \pi i |n|^2} f(x) e^{-2 \pi i \langle n, x \rangle} \, \textrm{d} x.
\end{equation}
Here $\partial A_j$ denotes the boundary of $A_j$, i.e.\ the union of its facets, and $\nu : \partial A_j \to \mathbb{R}^{d-1}$ is the outer unit normal vector. Since $f(x)$, and hence $f(x)e^{-2 \pi i \langle n, x \rangle}$ is periodic modulo $\mathbb{Z}^{d-1}$ and continuous, the sum of the right hand side of \eqref{divtheorem} over $1 \le j \le m$ is zero. Indeed, each facet appears twice in the sum, with the same integrand except with opposite signs because the outer normals are negatives of each other. Therefore summing \eqref{divtheorem} over $1 \le j \le m$ we obtain
\begin{equation}\label{sumAj}
\int_{[0,1]^{d-1}} f(x) e^{- 2 \pi i \langle n,x \rangle} \, \textrm{d}x = \sum_{j=1}^m \frac{\langle a_j, n \rangle}{2 \pi i |n|^2} \int_{A_j} e^{-2 \pi i \langle n,x \rangle} \, \textrm{d} x.
\end{equation}

The sum has $m=O(1)$ terms, thus it is enough to estimate the terms separately. Let $A=A_j \subseteq [0,1]^{d-1}$ for some $1 \le j \le m$. We follow the methods of Randol \cite{Randol} to bound the Fourier transform of the characteristic function of the polytope $A$. An ordered tuple $\mathcal{F}=(F_{d-1}, F_{d-2}, \dots, F_k)$ is called a flag of $A$ if $0 \le k \le d-1$, $F_{\ell}$ is an $\ell$-dimensional face of $A$ for every $k \le \ell \le d-1$, and $F_{d-1} \supset F_{d-2} \supset \cdots \supset F_k$. (Note $F_{d-1}=A$.) We call $\mathcal{F}$ a complete flag if $k=0$. Recall that $A$ has $O(1)$ facets, therefore the number of flags of $A$ is also $O(1)$.

To every given flag $\mathcal{F}=(F_{d-1}, F_{d-2}, \dots, F_k)$ let us associate orthogonal vectors $v_{d-2}, v_{d-3}, \dots, v_k$ such that $v_{\ell} \in \mathbb{R}^{d-1}$ is an outer normal vector of $F_{\ell}$ in the affine hull of $F_{\ell+1}$ for every $k \le \ell \le d-2$. Note that $v_{d-2}, \dots, v_k$ can be obtained by applying the Gram--Schmidt orthogonalization procedure to the normal vectors of certain facets of $A$, therefore we can also ensure that the coordinates of $v_{d-2}, \dots, v_k$ are all in $K$ (but the vectors might not have unit length). For every $k \le \ell \le d-1$ let $\pi_{\ell} : \mathbb{R}^{d-1} \to \mathbb{R}^{d-1}$ denote the orthogonal projection onto the $\ell$-dimensional linear subspace (i.e.\ containing the origin) parallel to $F_{\ell}$. In particular, for a complete flag we obtain an orthogonal basis $v_{d-2}, \dots, v_0$ of $\mathbb{R}^{d-1}$, defining linearly independent linear forms $L_1(x)=\langle v_{d-2}, x \rangle, \dots, L_{d-1}(x)= \langle v_0, x \rangle$ of the variables $x=(x_1, \dots, x_{d-1})$ with coefficients in $K$. Let $\mathcal{A}=\mathcal{A}_j$ denote the set of such linearly independent linear forms $(L_1, \dots, L_{d-1})$ associated to complete flags of $A=A_j$.

Clearly $|n|=|\pi_{d-1} (n)| \ge |\pi_{d-2}(n)| \ge \cdots \ge |\pi_k (n)|$. Let us call $\mathcal{F}$ a ``relevant flag'' if $|\pi_k (n)| < 1$ but $|\pi_{k+1}(n)| \ge 1$. We will express $\int_A e^{-2 \pi i \langle n, x \rangle} \, \textrm{d}x$ as a sum over all relevant flags of $A$. Formally, our integral is associated to the only flag of length 1, namely $(F_{d-1})$, which is not a relevant flag.

We use the following algorithm. Let us apply the divergence theorem to $F(x)=\frac{-n}{2 \pi i |n|^2} e^{-2 \pi i \langle n,x \rangle}$ on $A$. The integral over $\partial A$ can be written as a sum over all flags $(F_{d-1}, F_{d-2})$ of length 2, with terms
\begin{multline*}
\int_{F_{d-2}} \frac{-\langle v_{d-2}, n \rangle}{2 \pi i |v_{d-2}| |n|^2} e^{-2 \pi i \langle n,x \rangle} \, \textrm{d} x = \\ \frac{-\langle v_{d-2}, n \rangle}{2 \pi i |v_{d-2}| |n|^2} e^{-2 \pi i \langle n, w_{d-2} \rangle} \int_{\pi_{d-2}(F_{d-2})} e^{-2 \pi i \langle \pi_{d-2} (n) , x \rangle} \, \textrm{d} x.
\end{multline*}
Here $w_{d-2} \in \mathbb{R}^{d-1}$ is the vector for which $\pi_{d-2} (F_{d-2})+w_{d-2}=F_{d-2}$. The linear subspace containing $\pi_{d-2} (F_{d-2})$ can be isometrically identified with $\mathbb{R}^{d-2}$, thus $\langle \pi_{d-2}(n), x \rangle$ is preserved in this identification. This way we obtain
\[ \int_A e^{-2 \pi i \langle n, x \rangle} \, \textrm{d}x = \sum_{(F_{d-1}, F_{d-2})} C_n(F_{d-1}, F_{d-2}) \int_{\pi_{d-2} (F_{d-2})} e^{-2 \pi i \langle \pi_{d-2} (n), x \rangle} \, \textrm{d}x  \]
with some coefficients $|C_n (F_{d-1}, F_{d-2})| \le \frac{1}{2 \pi |\pi_{d-1}(n)|}$ (recall $\pi_{d-1}(n)=n$).

The terms indexed by relevant flags $(F_{d-1}, F_{d-2})$ are kept as they are. (Since $|\pi_{d-2}(n)|<1$, it is not worth applying the divergence theorem again.) If a term is indexed by a non-relevant flag $(F_{d-1}, F_{d-2})$, we apply the divergence theorem again and replace it by a sum over all extensions $(F_{d-1}, F_{d-2}, F_{d-3})$. We continue in a similar fashion: if a flag $(F_{d-1}, F_{d-2}, \dots, F_k)$ becomes relevant, we keep the corresponding term. If a flag is not relevant, we apply the divergence theorem again. The algorithm stops when every term in our sum is associated to a relevant flag. Note that since $|\pi_0 (n)|=0$, eventually every flag becomes relevant, and so the algorithm terminates. The algorithm yields a formula of the form
\begin{equation}\label{relevantflags}
\int_A e^{-2 \pi i \langle n, x \rangle} \, \textrm{d}x = \sum_{\substack{(F_{d-1}, F_{d-2}, \dots, F_k) \\ \textrm{relevant flags}}} C_n(F_{d-1}, F_{d-2}, \dots, F_k) \int_{\pi_k (F_k)} e^{-2 \pi i \langle \pi_k (n), x \rangle} \, \textrm{d}x
\end{equation}
with some coefficients $|C_n(F_{d-1}, F_{d-2}, \dots, F_k)| \le \prod_{\ell=k+1}^{d-1} \frac{1}{2 \pi |\pi_{\ell} (n)|}$.

Consider a relevant flag $(F_{d-1}, F_{d-2}, \dots, F_k)$. The corresponding integral on the right hand side of \eqref{relevantflags} is $O(1)$. If $k>0$, let us extend the relevant flag arbitrarily to a complete flag $(F_{d-1}, F_{d-2}, \dots, F_0)$. By the definition of a relevant flag we have $1>|\pi_k (n)| \ge |\pi_{k-1}(n)| \ge \cdots \ge |\pi_1 (n)|$, therefore $C_n (F_{d-1}, F_{d-2}, \dots, F_k) = O(1/\prod_{\ell=1}^{d-1} (|\pi_{\ell}(n)|+1))$. Clearly $|\pi_{\ell}(n)| \ge |\langle v_{\ell-1}, n \rangle| / |v_{\ell-1}|$ for every $1 \le \ell \le d-1$, hence we obtain the estimate
\[ \int_A e^{-2 \pi i \langle n, x \rangle} \, \textrm{d}x = O \left( \sum_{(L_1, \dots, L_{d-1}) \in \mathcal{A}} \frac{1}{\prod_{\ell=1}^{d-1} (|L_{\ell}(n)|+1)} \right) . \]
This holds for every $A=A_j$, therefore in light of \eqref{sumAj} $\mathcal{L}=\bigcup_{j=1}^m \mathcal{A}_j$ satisfies the claim of Lemma \ref{fourierlemma}.
\end{proof}

Note that for any linearly independent linear forms $L_1, \dots, L_{d-1}$ of $d-1$ variables we have
\[ \sum_{\substack{n \in \mathbb{Z}^{d-1} \\ n \neq 0}} \frac{1}{|n| \prod_{k=1}^{d-1} (|L_k(n)|+1)} < \infty . \]
Lemma \ref{fourierlemma} thus implies, in particular, that the Fourier series of $f$ is absolutely convergent. It follows (see e.g.\ \cite[Proposition 3.2.5.]{Grafakos}) that the Fourier series converges pointwise to $f$, i.e.\ $f(x)=\sum_{n \in \mathbb{Z}^{d-1}} \hat{f}(n) e^{2 \pi i \langle n,x \rangle}$ for every $x \in \mathbb{R}^{d-1}$, where $\hat{f}(n)=\int_{[0,1]^{d-1}} f(x)e^{-2\pi i \langle n,x \rangle} \, \mathrm{d}x$. It is not difficult to see from Fubini's theorem that
\[ \hat{f}(0)=\int_{[0,1]^{d-1}} f(x) \, \mathrm{d}x = \lambda (P). \]
Replacing $f$ by its Fourier series in \eqref{discretedynsyst}, and switching the order of summation we thus obtain with $s^*=(s_1, \dots, s_{d-1})$ and $\alpha^*=(\alpha_1, \dots, \alpha_{d-1})$ that
\[ \Delta_T (s, \alpha, P)=\sum_{k=0}^{T-1} \sum_{\substack{n \in \mathbb{Z}^{d-1}\\ n \neq 0}} \hat{f}(n) e^{2 \pi i \langle n, s^* +k \alpha^* \rangle} = \sum_{\substack{n \in \mathbb{Z}^{d-1}\\ n \neq 0}} \hat{f}(n) e^{2 \pi i \langle n, s^* \rangle} \frac{1-e^{2 \pi i \langle n, \alpha^* \rangle T}}{1-e^{2 \pi i \langle n, \alpha^* \rangle}} . \]
Using the general estimate $|1-e^{2 \pi i z}|=2|\sin (\pi z)| \ge 4 \left\| z \right\|$, $z \in \mathbb{R}$, we get
\begin{equation}\label{fourierbound}
|\Delta_T (s, \alpha, P)| \le \sum_{\substack{n \in \mathbb{Z}^{d-1}\\ n \neq 0}} |\hat{f}(n)| \cdot \frac{1}{2 \left\| n_1 \alpha_1 + \cdots + n_{d-1} \alpha_{d-1} \right\|} .
\end{equation}
In light of Lemma \ref{fourierlemma} it is thus enough to prove that for any linearly independent linear forms $L_1, \dots, L_{d-1}$ of $d-1$ variables with coefficients in $K$ we have
\begin{equation}\label{enoughtosee}
\sum_{\substack{n \in \mathbb{Z}^{d-1}\\ n \neq 0}} \frac{1}{|n| \prod_{k=1}^{d-1} (|L_k (n)|+1) \left\| n_1 \alpha_1 + \cdots + n_{d-1} \alpha_{d-1} \right\|} < \infty .
\end{equation}

We know that $\left\| n_1 \alpha_1 + \cdots + n_{d-1} \alpha_{d-1} \right\| \prod_{k=1}^{d-1} (|L_k (n)|+1) \ge C |n|^{-\gamma}$ for every $n \in \mathbb{Z}^{d-1}$, $n \neq 0$ with some constants $C>0$ and $\gamma <1$. For any integers $\ell_1, \dots, \ell_{d-1} \ge 0$ and $\ell \ge 0$ let $S_{\ell} (\ell_1, \dots, \ell_{d-1})$ denote the set of all $n \in \mathbb{Z}^{d-1}$, $n \neq 0$ such that $2^{\ell} \le |n| < 2^{\ell +1}$ and $2^{\ell_k} \le |L_k(n)|+1 < 2^{\ell_k +1}$ for all $1 \le k \le d-1$. Let $g: S_{\ell}(\ell_1, \dots, \ell_{d-1}) \to (-1/2, 1/2]$ be the function $g(n)=n_1 \alpha_1 + \cdots +n_{d-1} \alpha_{d-1} \pmod{1}$.

Let $H=\lceil C^{-1} 2^{(\ell_1+2) + \cdots + (\ell_{d-1}+2)} 2^{\gamma (\ell+2)} \rceil$. For every $n \in S_{\ell}(\ell_1, \dots, \ell_{d-1})$ we have
\[ |g(n)| = \left\| n_1 \alpha_1 + \cdots + n_{d-1} \alpha_{d-1} \right\| \ge \frac{1}{H} . \]
Moreover, for any $n,m \in S_{\ell}(\ell_1, \dots, \ell_{d-1})$, $n \neq m$ we have $|L_k (n-m)|+1 \le |L_k (n)| + |L_k (m)| +1 < 2^{\ell_k+2}$ for every $k$ and $|n-m| < 2^{\ell+2}$, and hence
\[ |g(n)-g(m)| \ge \left\| (n_1-m_1) \alpha_1 + \cdots + (n_{d-1}-m_{d-1}) \alpha_{d-1}  \right\| > \frac{1}{H} . \]
In other words, $g(n) \not\in (-1/H,1/H)$ for any $n$, and every interval of the form $[h/H, (h+1)/H)$ and $(-(h+1)/H,-h/H]$, $h \ge 1$ contains $g(n)$ for at most one $n$. Since $|g(n)| \le 1/2$, we therefore obtain
\[ \begin{split} \sum_{n \in S_{\ell} (\ell_1, \dots, \ell_{d-1})} \frac{1}{|g(n)|} \le 2 \sum_{1 \le h \le H/2} \frac{H}{h} &= O(H \log H) \\ &=O \left( 2^{\ell_1 + \cdots + \ell_{d-1}} 2^{\gamma \ell} (\ell_1 + \cdots + \ell_{d-1} + \ell) \right) , \end{split}\]
and consequently
\begin{multline*}
\sum_{n \in S_{\ell} (\ell_1, \dots, \ell_{d-1})} \frac{1}{|n| \prod_{k=1}^{d-1} (|L_k (n)|+1) \left\| n_1 \alpha_1 + \cdots + n_{d-1} \alpha_{d-1} \right\|} \\ = O \left( 2^{(\gamma -1)\ell} (\ell_1 + \cdots + \ell_{d-1} + \ell) \right) .
\end{multline*}

Note that $|L_k (n)|+1=O(|n|)$ shows $\ell_k = O(\ell)$, unless $S_{\ell} (\ell_1, \dots, \ell_{d-1})$ is empty. Summing over $0 \le \ell_1, \dots, \ell_{d-1} = O(\ell)$ we get
\[ \sum_{2^{\ell} \le |n| <2^{\ell+1}} \frac{1}{|n| \prod_{k=1}^{d-1} (|L_k (n)|+1) \left\| n_1 \alpha_1 + \cdots + n_{d-1} \alpha_{d-1} \right\|} = O \left( 2^{(\gamma -1)\ell} \ell^d \right) . \]
Finally, summing over $\ell \ge 0$ shows that \eqref{enoughtosee} indeed holds. The proof of Theorem \ref{maintheorem} is thus complete.
\end{proof}

\begin{proof}[Proof of Theorem \ref{theorem4}] We use the notation and follow the proof of Theorem \ref{maintheorem}. From the definition \eqref{DeltaT} of $\Delta_T (s, \alpha, P)$ it is easy to deduce that
\[ \begin{split} |\Delta_T (s, \alpha, P) - \Delta_{T+s_2} ((\{ s_1-\alpha_1 s_2\} ,0), \alpha, P)| &\le 1, \\ |\Delta_T (s, \alpha, P) - \Delta_{\lceil T \rceil} (s, \alpha, P)| &\le 1 . \end{split} \]
In other words, the error of replacing $s_2$ by $0$, and $T$ by $\lceil T \rceil$ is at most $2$. From now on we will assume $s_2=0$ and that $T \in \mathbb{N}$, and will prove the estimate in the claim without the first term $2$.

Let $f: \mathbb{R} \to \mathbb{R}$ be as in \eqref{fdefinition}, and for any $x \in [0,1]$ let $g_x(t)=(x+t\alpha_1, t)$, as before. Since $0<\alpha_1<1$, the line segment $g_x(t)$, $t \in [0,1]$ stays in $[0,2] \times [0,1]$, and so it can only intersect the translates $P$ and $P+(1,0)$. That is, for any $x \in [0,1]$
\begin{equation}\label{P,P+(1,0)}
f(x)=\int_0^1 \chi_P (g_x(t)) \, \mathrm{d}t + \int_0^1 \chi_{P+(1,0)} (g_x(t)) \, \mathrm{d}t .
\end{equation}
Again, $f$ is a piecewise linear function. Indeed, by applying a projection in the direction $\alpha$, that is, the map $\pi : \mathbb{R}^2 \to \mathbb{R}$, $\pi (x_1,x_2)=x_1-\alpha_1 x_2$ to the vertices of $P$ and $P+(1,0)$, we obtain a partition $0=c_0<c_1<\cdots<c_m=1$ of the interval $[0,1]$. (The projections outside $[0,1]$ are discarded.) Note that $m \le N+1$ since a pair of corresponding vertices of $P$ and $P+(1,0)$ have projections at distance $1$ from each other. For a given $1 \le j \le m$, as $x$ runs in $[c_{j-1},c_j]$ the line segment $g_x(t)$, $t \in [0,1]$ either does not intersect $P$, or intersects the same pair of edges $e_k, e_{\ell}$ of $P$ with some $1 \le k<\ell \le N$ depending on $j$. Thus the first term in \eqref{P,P+(1,0)} is of the form $a_j'x+b_j'$ on $[c_{j-1},c_j]$. As observed in \eqref{ajformula}, either $a_j'=0$ or
\[ |a_j'| = |\alpha| \left| \frac{\nu_{k,1}}{\langle \nu_k, \alpha \rangle} - \frac{\nu_{\ell,1}}{\langle \nu_{\ell}, \alpha \rangle} \right| = |\alpha| \left| \frac{\nu_{k,1} \nu_{\ell,2} -\nu_{k,2} \nu_{\ell,1} }{\langle \nu_k, \alpha \rangle \cdot \langle \nu_{\ell}, \alpha \rangle} \right| ,\]
where $\nu_k=(\nu_{k,1}, \nu_{k,2}), \nu_{\ell}=(\nu_{\ell,1},\nu_{\ell,2})$ are normal vectors of $e_k$, $e_{\ell}$, respectively. Using the angles $\phi_k, \phi_{\ell}$ in the latter case we have
\[ |a_j'|= |\alpha| \frac{|\nu_k| \cdot |\nu_{\ell}| \cdot |\sin (\phi_k - \phi_{\ell})|}{|\nu_k| \cdot |\alpha| \cdot |\cos \left( \frac{\pi}{2}-\phi_k \right)| \cdot |\nu_{\ell}| \cdot |\alpha| \cdot |\cos \left( \frac{\pi}{2}-\phi_{\ell} \right)|} = \frac{\left| \cot \phi_k - \cot \phi_{\ell} \right|}{|\alpha|} . \]
Note that although the angles formed by $\alpha$, $\nu_k$ and $\nu_{\ell}$ are not well-defined functions of $\phi_k, \phi_{\ell}$, the absolute value of the trigonometric functions in the formula above are well-defined. Similarly, the second term in \eqref{P,P+(1,0)} is of the form $a_j''x+b_j''$ on $[c_{j-1},c_j]$ with either $a_j''=0$ or $|a_j''|=|\cot \phi_p- \cot \phi_q|/|\alpha|$ with some $1 \le p<q\le N$ depending on $j$.

Thus $f(x)$ is of the form $f(x)=a_jx+b_j$ on $[c_{j-1},c_j]$, where $a_j=a_j'+a_j''$. Consider the Fourier coefficients $\hat{f}(n)=\int_0^1 f(x)e^{-2 \pi i n x} \, \mathrm{d}x$, $n \in \mathbb{Z}$. It is easy to see from Fubini's theorem that $\hat{f}(0)=\lambda (P)$. For $n \neq 0$ we can apply integration by parts to obtain
\[ \begin{split} \hat{f}(n) &= \sum_{j=1}^m \left( \int_{c_{j-1}}^{c_j} (a_jx+b_j)e^{-2 \pi i nx} \, \mathrm{d}x \right) \\ &= \sum_{j=1}^m \left( f(c_j)\frac{e^{-2 \pi i nc_j}}{-2 \pi i n}-f(c_{j-1}) \frac{e^{-2 \pi i nc_{j-1}}}{-2 \pi i n} \right) -\sum_{j=1}^m a_j\int_{c_{j-1}}^{c_j} \frac{e^{-2 \pi i n x}}{-2 \pi i n} \, \mathrm{d}x . \end{split} \]
Here the first sum is $0$ because $f$ is continuous and $1$-periodic, hence
\[ |\hat{f}(n)| \le \sum_{j=1}^m \frac{|a_j'| + |a_j''|}{2 \pi^2 n^2} \le \frac{N+1}{\pi^2 |\alpha| n^2} \max_{1 \le k<\ell \le N} \left| \cot \phi_k - \cot \phi_{\ell} \right| . \]
From \eqref{fourierbound} we finally deduce
\[ |\Delta_T (s, \alpha, P)| \le \sum_{\substack{n \in \mathbb{Z} \\ n \neq 0}} \frac{|\hat{f}(n)|}{2 \| n \alpha_1 \|} \le \frac{N+1}{\pi^2 |\alpha |} \max_{1 \le k<\ell \le N} \left| \cot \phi_k - \cot \phi_{\ell} \right| \sum_{n=1}^{\infty} \frac{1}{n^2 \| n \alpha_1 \|} . \]
\end{proof}

\section{The proof of Theorem \ref{theorem1}}\label{section3}

We now prove Theorem \ref{theorem1}. By applying a simple integral transformation in the definition \eqref{DeltaT} of $\Delta_T(s, \alpha, P)$, we may assume $\alpha_d=1$. Choosing $K$ to be the field of algebraic reals, it is thus enough to show that if $\alpha_1, \ldots, \alpha_{d-1},1$ are algebraic and linearly independent over $\mathbb{Q}$, then $\alpha$ satisfies the Diophantine condition of Theorem \ref{maintheorem}. The celebrated subspace theorem of Schmidt \cite{Schmidt} shows that this Diophantine condition is in fact satisfied with any $\gamma >0$. In other words, we do not even need the full power of the subspace theorem. Unfortunately, most monographs on simultaneous Diophantine approximation prove this condition only for the linear forms $L_1(x)=x_1, \dots, L_{d-1}(x)=x_{d-1}$. For the sake of completeness, we include a proof for arbitrary linearly independent linear forms with real algebraic coefficients. Nevertheless, the following theorem can still be considered to be a form of the subspace theorem of Schmidt.

\begin{thm}[Schmidt] Let $d \ge 2$, and let the algebraic reals $\alpha_1, \dots, \alpha_{d-1}, 1$ be linearly independent over $\mathbb{Q}$. Let $L_1, \dots, L_{d-1}$ be linearly independent linear forms of $d-1$ variables with real algebraic coefficients. For any $\varepsilon >0$ the inequality
\begin{equation}\label{Schmidtbound}
\left\| \alpha_1 n_1 + \cdots + \alpha_{d-1} n_{d-1} \right\| \cdot \prod_{k=1}^{d-1} \left( |L_k (n)|+1 \right) < |n|^{-\varepsilon}
\end{equation}
has finitely many integer solutions $n \in \mathbb{Z}^{d-1}$.
\end{thm}

\begin{proof} We derive the theorem from two different versions of Schmidt's subspace theorem. First, a special case of the subspace theorem \cite[Corollary 1]{Schmidt} says that for any $\varepsilon >0$ the inequality $\left\| \alpha_1 n_1 + \cdots + \alpha_{d-1} n_{d-1} \right\| < |n|^{-(d-1)-\varepsilon}$ has finitely many integer solutions $n \in \mathbb{Z}^{d-1}$. Therefore it will be enough to consider $n \in \mathbb{Z}^{d-1}$ such that, say, $\left\| \alpha_1 n_1 + \dots + \alpha_{d-1} n_{d-1} \right\| \ge |n|^{-d}$.

Let $c_0 \le c_1 \le \dots \le c_{d-1}$ be reals such that $\sum_{k=0}^{d-1} c_k =0$, and let $M_0, M_1, \dots$, $M_{d-1}$ be linear forms of $d$ variables with real algebraic coefficients. We call
\begin{equation}\label{generalroth}
( M_0,M_1, \dots, M_{d-1} ; c_0, c_1, \dots , c_{d-1} )
\end{equation}
a \textit{general Roth system} if for every $\delta >0$ there exists a $Q_1>0$ such that for any real $Q \ge Q_1$ the system of inequalities
\begin{equation}\label{systemofineq}
|M_k (m)| \le Q^{c_k-\delta} \qquad (0 \le k \le d-1)
\end{equation}
has no integer solution $m \in \mathbb{Z}^d, m \neq 0$. For a linear subspace $S$ of $\mathbb{R}^d$ of dimension $r>0$, define $c(S)$ the following way. If the rank of the forms $M_0, M_1, \dots, M_{d-1}$ on $S$ is less than $r$, then let $c(S)=\infty$. Otherwise, let $k_1$ be the smallest index such that $M_{k_1}$ is not constant zero on $S$. Let $k_2>k_1$ be the smallest index such that $M_{k_1}, M_{k_2}$ have rank 2 on $S$ etc, and define $c(S)=c_{k_1} + \cdots + c_{k_r}$. A general version of the subspace theorem \cite[Theorem 2]{Schmidt} states that \eqref{generalroth} is a general Roth system if and only if $c(S) \le 0$ for every rational linear subspace $S \neq 0$ of $\mathbb{R}^d$.

Fix an $\varepsilon >0$, and let us choose a positive integer $p$ such that $1/p < \varepsilon /(3d^2)$. Let $M_0 (x) = x_0+ \alpha_1 x_1 + \cdots + \alpha_{d-1} x_{d-1}$, and let $M_k (x) = L_k (x_1, \dots, x_{d-1})$, $1 \le k \le d-1$ be linear forms of the variables $x=(x_0, x_1, \dots, x_{d-1})$. We wish to apply the subspace theorem to $M_0, M_1, \dots, M_{d-1}$ with $\delta = 1/p$, $c_0=-1$ and $c_1, \dots, c_{d-1}$ all of the form $j/p$ for some integer $1 \le j \le p$ such that $\sum_{k=0}^{d-1}c_k=0$. The forms $M_0, M_1, \dots, M_{d-1}$ are clearly linearly independent, because $x_0$ appears only in $M_0$, and $L_1, \dots, L_{d-1}$ are linearly independent. Hence on any rational subspace $S$ of $\mathbb{R}^d$ of dimension $r>0$ the rank of $M_0, M_1, \dots, M_{d-1}$ is $r$. Moreover, choosing a nonzero rational vector $v \in S$ we have $M_0 (v) \neq 0$. Therefore in the definition of $c(S)$ we have $k_1=0$, and so
\[ c(S) = c_{k_1} + \dots + c_{k_r} \le -1 + \sum_{k=1}^{d-1} c_k =0 . \]
According to the subspace theorem we thus have a general Roth system. Since there are finitely many ways to choose such $c_1, \dots, c_{d-1}$, there exists a $Q_1 >0$ depending only on $\alpha_1, \dots, \alpha_{d-1}$, $L_1, \dots, L_{d-1}$ and $\varepsilon$ such that for any real $Q \ge Q_1$ and any such $c_1, \dots, c_{d-1}$ the system of inequalities \eqref{systemofineq} has no integral solution $m \in \mathbb{Z}^d$, $m \neq 0$.

Consider now an integer solution $n \in \mathbb{Z}^{d-1}$, $n \neq 0$ of \eqref{Schmidtbound} such that
\[ \left\| \alpha_1 n_1 + \cdots + \alpha_{d-1} n_{d-1} \right\| \ge |n|^{-d}.\]
Let the integer $Q>0$ be such that
\[ \frac{1}{(Q+1)^{1+\delta}} < \left\| \alpha_1 n_1 + \cdots + \alpha_{d-1} n_{d-1} \right\| \le \frac{1}{Q^{1+\delta}} . \]
From \eqref{Schmidtbound} we have $1/(Q+1)^{1+\delta} \le |n|^{-\varepsilon}$, and so for a given integer $Q>0$ there are finitely many such solutions $n$. It will therefore be enough to show that $Q < Q_1$.

Choosing $m_k=n_k$ for $1 \le k \le d-1$ and $m_0$ to be the integer closest to $\alpha_1 n_1 + \cdots + \alpha_{d-1} n_{d-1}$, we have $|M_0 (m)| \le Q^{c_0-\delta}$. Note $Q \le |n|^d$. From \eqref{Schmidtbound} we have
\[ -(1+\delta) \log (Q+1) + \sum_{k=1}^{d-1} \log (|L_k (n)| +1) < - \varepsilon \log |n| \le - \frac{\varepsilon}{d} \log Q . \]
Since $1+\delta - \varepsilon /d < 1+\varepsilon / (3d^2) -\varepsilon /d$, we have, for $Q$ large enough, that
\[ \sum_{k=1}^{d-1} \frac{\log (|L_k (n)|+1)}{\log Q} < 1+\frac{\varepsilon}{3d^2} -\frac{\varepsilon}{d} . \]
Let $c_k'$ be the number of the form $j/p$ for some integer $j \ge 1$, such that $c_k'-2/p < \frac{\log (|L_k (n)|+1)}{\log Q} \le c_k'-1/p$. Then we clearly have
\[ \sum_{k=1}^{d-1} c_k' \le \sum_{k=1}^{d-1} \left( \frac{\log (|L_k (n)|+1)}{\log Q} + \frac{2}{p} \right) < 1+\frac{\varepsilon}{3d^2} -\frac{\varepsilon}{d} + \frac{2d}{p} <1 . \]
By increasing $c_k'$ we can find numbers $c_k \ge c_k'$ of the form $j/p$ for some integer $1 \le j \le p$ such that $\sum_{k=0}^{d-1} c_k = -1 + \sum_{k=1}^{d-1} c_k =0$. From $\frac{\log (|L_k (n)|+1)}{\log Q} \le c_k'-1/p$ we have
\[ |M_k (m)| \le |L_k (n)|+1 \le Q^{c_k' - \delta} \le Q^{c_k - \delta} \qquad (1 \le k \le d-1) . \]
Therefore $Q < Q_1$, and we are done.
\end{proof}

\end{document}